\documentclass{article}
\usepackage{amssymb,amsmath,amsthm,tikz}

\usepackage[normalem]{ulem}
\usepackage{cancel}

\theoremstyle{theorem}
\newtheorem{theorem}{Theorem}[section]

\newtheorem{corollary}[theorem]{Corollary}
\newtheorem{conjecture}[theorem]{Conjecture}

\theoremstyle{definition}
\newtheorem{definition}[theorem]{Definition}

\DeclareMathOperator{\crank}{crank}
\DeclareMathOperator{\mex}{mex}

\title{On Blecher and Knopfmacher's Fixed Points for Integer Partitions}
\author{Brian Hopkins and James A.\ Sellers}

\begin{document}

\maketitle

\begin{abstract}
Recently, Blecher and Knopfmacher explored the notion of fixed points in integer partitions and hypothesized on the relative number of partitions with and without a fixed point.  We resolve their open question by working fixed points into a growing number of interconnected partition statistics involving Frobenius symbols, Dyson's crank, and the mex (minimal excluded part).  Also, we generalize the definition of fixed points and connect that expanded notion to the $\mex_j$ defined by Hopkins, Sellers, and Stanton as well as the $j$-Durfee rectangle defined by Hopkins, Sellers, and Yee.
\end{abstract}

\section{Introduction}

Given a positive integer $n$, a partition of $n$ is a collection of positive integers $\lambda = (\lambda_1, \ldots, \lambda_j)$ with $\sum \lambda_i = n$.  We use the notation $\lambda \vdash n$ to indicate that $\lambda$ is a partition of $n$.  The $\lambda_i$, called parts, are ordered so that $\lambda_1 \ge \cdots \ge \lambda_j$.  Write $p(n)$ for the number of partitions of $n$.

Recently, Blecher and Knopfmacher introduced the idea of fixed points in integer partitions \cite[Section 4]{bk}.

\begin{definition} \label{fp}
A partition $\lambda$ has a fixed point if there is an index $i$ for which $\lambda_i = i$.
\end{definition}

For example, among the partitions of 15, $\alpha=(5,3,3,3,1)$ has a fixed point since $\alpha_3 = 3$ while $\beta = (4,4,4,2,1)$ has no fixed points.

Since we require the parts of a partitions to be in nonincreasing order, a partition has at most one fixed point.  Let $f(n)$ be the number of partitions of $n$ with a fixed point, $g(n)$ the number without.  The authors hypothesized the following relation between these counts.

\begin{conjecture}[Blecher--Knopfmacher] \label{bkconj}
For all $n > 2$, there are more partitions of $n$ without a fixed point than partitions of $n$ with a fixed point.  That is, $g(n) > f(n)$ for $n > 2$.
\end{conjecture}

(Actually, they state the conjecture for $n \ge 2$, but since the partition $(2)$ does not have a fixed point and $(1,1)$ does, $g(2) = 1 = f(2)$.)

In the next section, we confirm their (corrected) conjecture by connecting the fixed point to an increasing collection of interrelated partition statistics defined below.  We begin with two concepts from the 19th century.

\begin{definition}
Given a partition $\lambda$, let $d(\lambda)$ be the number of parts such that $\lambda_i \ge i$.  This parameter is the side length of the Durfee square, the largest square of dots contained in the Ferrers diagram of $\lambda$, a visual representation of the partition where row $i$ has $\lambda_i$ dots.
\end{definition}

For our example partitions mentioned earlier, $d(\alpha) = d(\beta) = 3$.  See Figure \ref{fig1} for the Ferrers diagrams and Durfee squares of these partitions.

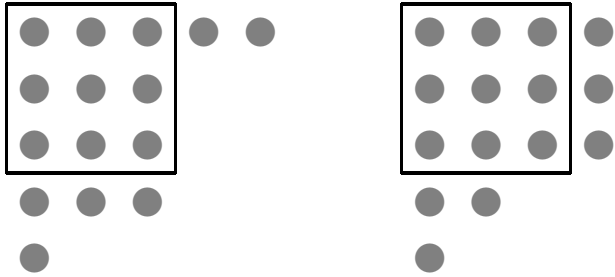
\begin{figure}
\begin{center}
\setlength{\unitlength}{.75cm}
\begin{picture}(11,5)
\thicklines
\put(0,4){{\color{gray} \circle*{.5}}}	\put(1,4){{\color{gray} \circle*{.5}}}	 \put(2,4){{\color{gray}\circle*{.5}}}	\put(3,4){{\color{gray}\circle*{.5}}} 	\put(4,4){{\color{gray}\circle*{.5}}}
\put(0,3){{\color{gray} \circle*{.5}}}	\put(1,3){{\color{gray} \circle*{.5}}}	 \put(2,3){{\color{gray}\circle*{.5}}}		
\put(0,2){{\color{gray} \circle*{.5}}}	\put(1,2){{\color{gray} \circle*{.5}}}	 \put(2,2){{\color{gray}\circle*{.5}}}
\put(0,1){{\color{gray} \circle*{.5}}}	\put(1,1){{\color{gray} \circle*{.5}}}	 \put(2,1){{\color{gray}\circle*{.5}}}
\put(0,0){{\color{gray}\circle*{.5}}}
\multiput(-0.5,1.5)(3,0){2}{\line(0,1){3}}
\multiput(-0.5,1.5)(0,3){2}{\line(1,0){3}}

\put(7,4){{\color{gray} \circle*{.5}}}	\put(8,4){{\color{gray} \circle*{.5}}}	\put(9,4){{\color{gray}\circle*{.5}}}	\put(10,4){{\color{gray}\circle*{.5}}}
\put(7,3){{\color{gray} \circle*{.5}}}	\put(8,3){{\color{gray} \circle*{.5}}}	\put(9,3){{\color{gray}\circle*{.5}}}	\put(10,3){{\color{gray}\circle*{.5}}}	
\put(7,2){{\color{gray} \circle*{.5}}}	\put(8,2){{\color{gray} \circle*{.5}}}	\put(9,2){{\color{gray}\circle*{.5}}}	\put(10,2){{\color{gray}\circle*{.5}}}
\put(7,1){{\color{gray} \circle*{.5}}}	\put(8,1){{\color{gray}\circle*{.5}}}		 
\put(7,0){{\color{gray}\circle*{.5}}}	
\multiput(6.5,1.5)(3,0){2}{\line(0,1){3}}
\multiput(6.5,1.5)(0,3){2}{\line(1,0){3}}
\end{picture}
\end{center}
\caption{The Ferrers diagrams of $\alpha=(5,3,3,3,1)$ and $\beta = (4,4,4,2,1)$ which both have a $3 \times 3$ Durfee square as shown.}  \label{fig1}
\end{figure}

\begin{definition} 
Given a partition $\lambda$, its Frobenius symbol is a $2 \times d(\lambda)$ array of nonnegative integers 
$$
\begin{pmatrix} 
a_1 & a_2 &  \cdots & a_{d(\lambda)} \\ 
b_1 & b_2 & \cdots & a_{d(\lambda)} 
\end{pmatrix}
$$
with $a_1>a_2>\dots > a_{d(\lambda)} \geq 0$ and $b_1>b_2>\dots > b_{d(\lambda)} \geq 0$
wherein, referencing the Ferrers diagram of $\lambda$, each $a_i$ indicates the number of dots in row $i$ to the right of the diagonal of the Durfee square and each $b_i$ indicates the number of dots in column $i$ below the diagonal.
\end{definition}

The Frobenius symbols for our examples are $\begin{pmatrix} 4 & 1 & 0 \\ 4 & 2 & 1 \end{pmatrix}$ for $\alpha$ and $\begin{pmatrix} 3 & 2 & 1 \\ 4 & 2 & 0 \end{pmatrix}$ for $\beta$; see Figure \ref{fig2}.

\begin{figure}
\begin{center}
\setlength{\unitlength}{.75cm}
\begin{picture}(11,4)
\thicklines
\put(0,4){{\color{gray} \circle*{.5}}}	\put(1,4){{\color{gray} \circle*{.5}}}	 \put(2,4){{\color{gray}\circle*{.5}}}	\put(3,4){{\color{gray}\circle*{.5}}} 	\put(4,4){{\color{gray}\circle*{.5}}}
\put(0,3){{\color{gray} \circle*{.5}}}	\put(1,3){{\color{gray} \circle*{.5}}}	 \put(2,3){{\color{gray}\circle*{.5}}}		
\put(0,2){{\color{gray} \circle*{.5}}}	\put(1,2){{\color{gray} \circle*{.5}}}	 \put(2,2){{\color{gray}\circle*{.5}}}
\put(0,1){{\color{gray} \circle*{.5}}}	\put(1,1){{\color{gray} \circle*{.5}}}	 \put(2,1){{\color{gray}\circle*{.5}}}
\put(0,0){{\color{gray}\circle*{.5}}}
\multiput(0.6,3.6)(3.8,0){2}{\line(0,1){0.8}}
\multiput(0.6,3.6)(0,0.8){2}{\line(1,0){3.8}}
\multiput(1.6,2.6)(0.8,0){2}{\line(0,1){0.8}}
\multiput(1.6,2.6)(0,0.8){2}{\line(1,0){0.8}}
\put(2.6,1.6){\line(0,1){0.8}}
\multiput(-0.4,-0.4)(0.8,0){2}{\line(0,1){3.8}}
\multiput(-0.4,-0.4)(0,3.8){2}{\line(1,0){0.8}}
\multiput(0.6,0.6)(0.8,0){2}{\line(0,1){1.8}}
\multiput(0.6,0.6)(0,1.8){2}{\line(1,0){0.8}}
\multiput(1.6,0.6)(0.8,0){2}{\line(0,1){0.8}}
\multiput(1.6,0.6)(0,0.8){2}{\line(1,0){0.8}}

\put(7,4){{\color{gray} \circle*{.5}}}	\put(8,4){{\color{gray} \circle*{.5}}}	\put(9,4){{\color{gray}\circle*{.5}}}	\put(10,4){{\color{gray}\circle*{.5}}}
\put(7,3){{\color{gray} \circle*{.5}}}	\put(8,3){{\color{gray} \circle*{.5}}}	\put(9,3){{\color{gray}\circle*{.5}}}	\put(10,3){{\color{gray}\circle*{.5}}}	
\put(7,2){{\color{gray} \circle*{.5}}}	\put(8,2){{\color{gray} \circle*{.5}}}	\put(9,2){{\color{gray}\circle*{.5}}}	\put(10,2){{\color{gray}\circle*{.5}}}
\put(7,1){{\color{gray} \circle*{.5}}}	\put(8,1){{\color{gray}\circle*{.5}}}		 
\put(7,0){{\color{gray}\circle*{.5}}}	
\multiput(7.6,3.6)(2.8,0){2}{\line(0,1){0.8}}
\multiput(7.6,3.6)(0,0.8){2}{\line(1,0){2.8}}
\multiput(8.6,2.6)(1.8,0){2}{\line(0,1){0.8}}
\multiput(8.6,2.6)(0,0.8){2}{\line(1,0){1.8}}
\multiput(9.6,1.6)(0.8,0){2}{\line(0,1){0.8}}
\multiput(9.6,1.6)(0,0.8){2}{\line(1,0){0.8}}
\multiput(6.6,-0.4)(0.8,0){2}{\line(0,1){3.8}}
\multiput(6.6,-0.4)(0,3.8){2}{\line(1,0){0.8}}
\multiput(7.6,0.6)(0.8,0){2}{\line(0,1){1.8}}
\multiput(7.6,0.6)(0,1.8){2}{\line(1,0){0.8}}
\put(8.6,1.4){\line(1,0){0.8}}
\end{picture}
\end{center}
\caption{Determining the Frobenius symbols for $\alpha=(5,3,3,3,1)$ and $\beta = (4,4,4,2,1)$.}  \label{fig2}
\end{figure}
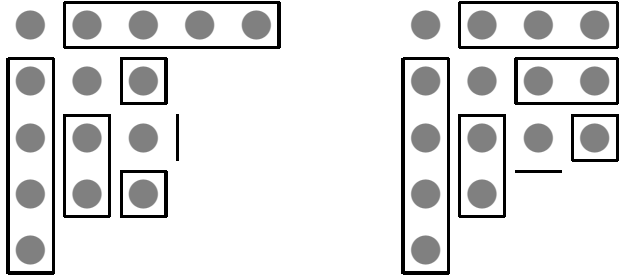

The next statistic is one of the most important in the study of integer partitions.  The crank, related to the Ramanujan congruences, was requested by Dyson in 1944 \cite{d} and eventually found by Andrews and Garvan in 1988  \cite{ag}.

\begin{definition}
For a partition $\lambda$, let $\omega(\lambda)$ be the number of parts 1 in $\lambda$ and $\mu(\lambda)$ the number of parts greater than $\omega(\lambda)$.  The crank of $\lambda$ is 
\[\crank(\lambda) =\begin{cases} \lambda_1 & \text{ if $\omega(\lambda)=0$},\\
\mu(\lambda) - \omega(\lambda) & \text{ if $\omega(\lambda)>0$}.
\end{cases}\]
\end{definition}

For our examples, $\crank(\alpha) = \crank(\beta) = 4 - 1 = 3$.  We use the notation $M(m,n)$ for the number of partitions of $n$ with crank $m$.

A newer and very simple partition statistic is the mex (from minimal excludant).

\begin{definition}
Given a partition $\lambda$, the mex is the smallest positive integer that is not a part of $\lambda$.
\end{definition}

For our examples, $\mex(\alpha)= 2$ since $(5,3,3,3,1)$ has a part 1 but no part 2 and $\mex(\beta) = 3$ since $(4,4,4,2,1)$ has parts 1 and 2 but no part 3.

In later sections, we will generalize the idea of a fixed point.  Those results use two concepts introduced in recent work.  First, we define a generalization of the mex given by  Hopkins, Sellers, and Stanton \cite[Definition 1]{hss}.

\begin{definition}
Given a partition $\lambda$ and a nonnegative integer $j$, let $\mex_j(\lambda)$ be the smallest integer greater than $j$ that is not a part of $\lambda$.
\end{definition}

For example, $\alpha=(5,3,3,3,1)$ has $\mex_1(\alpha) = 2$ and $\mex_2(\alpha) = 4$.  The unindexed mex corresponds to the $j=0$ case.

Note that this differs from \cite[Definition 1]{hss}, which also requires that $j$ be a part of $\lambda$.  Accordingly, here is a revised version of \cite[Theorem 2]{hss} connecting the $\mex_j$ and partitions with bounded crank.

\begin{theorem}[Hopkins, Sellers, Stanton] \label{hss2}
Given $n \ge 2$ and $j \ge 0$,
\begin{align*}
\sum_{m \ge j} M(m,n) & = \# \{ \mu \vdash n \mid \mex_j - j \equiv 1 \bmod 2 \text{ and $j$ is a part of $\mu$}\} \\
& = \# \{ \nu \vdash n-j \mid \mex_j - j \equiv 1 \bmod 2 \} 
\end{align*}
\end{theorem}

The first equality is the original result.  The second equality follows by removing a part $j$ from $\mu$.

The second concept we will use is a generalization of the Durfee square defined by Hopkins, Sellers, and Yee \cite[page 3]{hsy}.

\begin{definition}
Given a partition $\lambda$ and an integer $j$, let $d_j(\lambda)$ be the number of parts such that $\lambda_i \ge i+j$.  The $j$-Durfee rectangle, the largest $d \times (d+j)$ rectangle contained in the Ferrers diagram of $\lambda$, has $d_j(\lambda)$ rows and $d_j(\lambda)+j$ columns.
\end{definition}

For example, $\alpha=(5,3,3,3,1)$ has 1-Durfee rectangle $2 \times 3$ and $-2$-Durfee rectangle $4 \times 2$; see Figure \ref{fig3}.  The Durfee square corresponds to the $j=0$ case.

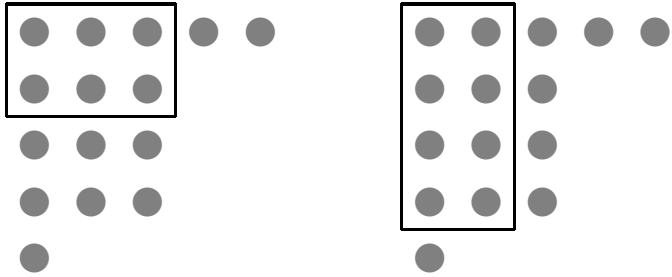
\begin{figure}
\begin{center}
\setlength{\unitlength}{.75cm}
\begin{picture}(11,5)
\thicklines
\put(0,4){{\color{gray} \circle*{.5}}}	\put(1,4){{\color{gray} \circle*{.5}}}	 \put(2,4){{\color{gray}\circle*{.5}}}	\put(3,4){{\color{gray}\circle*{.5}}} 	\put(4,4){{\color{gray}\circle*{.5}}}
\put(0,3){{\color{gray} \circle*{.5}}}	\put(1,3){{\color{gray} \circle*{.5}}}	 \put(2,3){{\color{gray}\circle*{.5}}}		
\put(0,2){{\color{gray} \circle*{.5}}}	\put(1,2){{\color{gray} \circle*{.5}}}	 \put(2,2){{\color{gray}\circle*{.5}}}
\put(0,1){{\color{gray} \circle*{.5}}}	\put(1,1){{\color{gray} \circle*{.5}}}	 \put(2,1){{\color{gray}\circle*{.5}}}
\put(0,0){{\color{gray}\circle*{.5}}}
\multiput(-0.5,2.5)(3,0){2}{\line(0,1){2}}
\multiput(-0.5,2.5)(0,2){2}{\line(1,0){3}}

\put(7,4){{\color{gray} \circle*{.5}}}	\put(8,4){{\color{gray} \circle*{.5}}}	 \put(9,4){{\color{gray}\circle*{.5}}}	\put(10,4){{\color{gray}\circle*{.5}}} 	\put(11,4){{\color{gray}\circle*{.5}}}
\put(7,3){{\color{gray} \circle*{.5}}}	\put(8,3){{\color{gray} \circle*{.5}}}	 \put(9,3){{\color{gray}\circle*{.5}}}		
\put(7,2){{\color{gray} \circle*{.5}}}	\put(8,2){{\color{gray} \circle*{.5}}}	 \put(9,2){{\color{gray}\circle*{.5}}}
\put(7,1){{\color{gray} \circle*{.5}}}	\put(8,1){{\color{gray} \circle*{.5}}}	 \put(9,1){{\color{gray}\circle*{.5}}}
\put(7,0){{\color{gray}\circle*{.5}}}
\multiput(6.5,0.5)(2,0){2}{\line(0,1){4}}
\multiput(6.5,0.5)(0,4){2}{\line(1,0){2}}
\end{picture}
\end{center}
\caption{The 1-Durfee rectangle and $-2$-Durfee rectangle of $\alpha=(5,3,3,3,1)$.}  \label{fig3}
\end{figure}

Many of our proofs use generating functions.  We recall the usual notation for the Pochhammer symbol: For $n\geq 1$, \[(a;q)_n = (1-a)(1-aq)\cdots(1-aq^{n-1}) \text{ and } 
(a;q)_\infty = \lim_{n \rightarrow \infty} (a;q)_n.\]  

The $j$-Durfee rectangle was instrumental in proving a new generating function for partitions with bounded crank \cite[Theorem 8]{hsy}.  For completeness, we repeat that result here.

\begin{theorem}[Hopkins, Sellers, Yee] \label{hsy8}
For any integer $j$,
\[ \sum_{m\ge j} \sum_{n\ge 0} M(m,n) q^n =\sum_{i\ge 0} \frac{q^{(i+1)(i+j)}}{(q;q)_i (q;q)_{i+j}}. \]
\end{theorem}

In the next section, we easily confirm Conjecture \ref{bkconj} by connecting the fixed point statistic to the Frobenius symbol and recalling recent work on the Frobenius symbol, mex, and crank.  In sections 3 and 4, we generalize the notion of fixed point and see that a more precise version of the conjecture of Blecher and Knopfmacher is the first case of a family of enumeration results for the number of partitions of $n$ with crank in a range of values.  Section 5 offers some ideas for future work.


\section{Confirming the conjecture of Blecher and Knopfmacher}

We confirm Conjecture \ref{bkconj} by connecting the fixed point partition statistic with the Frobenius symbol, the mex, and the crank.  
We use $\#S$ to denote the number of elements in the set $S$.

\begin{theorem}  \label{ffmc}
\renewcommand{\theenumi}{(\roman{enumi})}
For each $n \ge 2$, the following quantities are equal.
\begin{enumerate}
\item $f(n)$, the number of partitions of $n$ with a fixed point, \label{ffixed}
\item $\#\{\lambda \vdash n \mid \text{the top row of the Frobenius symbol includes 0}\}$, \label{fFrob}
\item $\#\{ \lambda \vdash n \mid \mex(\lambda) \equiv 0 \bmod 2\}$, \label{fmex}
\item $\displaystyle \sum_{m \ge 1} M(m,n)$. \label{fcrank}
\end{enumerate}
Under the same condition, the following quantities are also equal.
\begin{enumerate}
\renewcommand{\theenumi}{(\roman{enumi})}
\setcounter{enumi}{4}
\item $g(n)$, the number of partitions of $n$ without a fixed point, \label{gfixed}
\item $\#\{\lambda \vdash n \mid \text{the top row of the Frobenius symbol excludes 0}\}$, \label{gFrob}
\item $\#\{ \lambda \vdash n \mid \mex(\lambda) \equiv 1 \bmod 2\}$, \label{gmex}
\item $\displaystyle \sum_{m \ge 0} M(m,n)$. \label{gcrank}
\end{enumerate}
\end{theorem}

\begin{proof}
To connect {\it \ref{ffixed}} and {\it \ref{fFrob}}, suppose $\lambda$ has $\lambda_i = i$.  Since $\lambda_j \ge i$ for any $j < i$ and $\lambda_j \le i$ for any $j > i$, this means that the Durfee square length $d(\lambda) = i$.  Thus the last entry in the top row of the Frobenius symbol for $\lambda$ is $\lambda_i - i = 0$.  

Conversely, suppose a partition $\lambda$ with $d(\lambda) = i$ has Frobenius symbol including 0 in the top row.  The 0 is necessarily in the last position of the top row and thus $\lambda_i = i$, so $\lambda$ is among the partitions counted by $f(n)$.

The equality of {\it \ref{gfixed}}  and {\it \ref{gFrob}} follows by complementarity.

The other results all follow from previous work.  In particular, Andrews \cite[Theorem 4]{a11} shows that {\it \ref{gFrob}}  and {\it \ref{gmex}}  are equal, so {\it \ref{fFrob}}  and {\it \ref{fmex}} follow from complementarity.  Hopkins and Sellers \cite[Corollary 2]{hs} show that {\it \ref{fmex}}  and {\it \ref{fcrank}}  are equal (see also the contemporaneous independent result of Andrews and Newman \cite[Theorem 2]{an20}).  Finally, the equality of {\it \ref{gmex}}  and {\it \ref{gcrank}} uses complementarity and also the crank symmetry result of Andrews and Garvan \cite[Equation 1.9]{ag}, namely $M(m,n) = M(-m,n)$.
\end{proof}

See Figure \ref{fig4} for examples relating fixed points to the Durfee square.

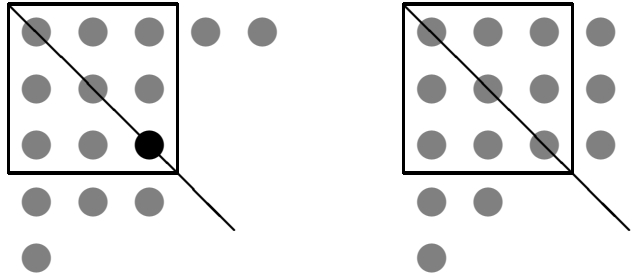
\begin{figure}
\begin{center}
\setlength{\unitlength}{.75cm}
\begin{picture}(11,5)
\thicklines
\put(0,4){{\color{gray} \circle*{.5}}}	\put(1,4){{\color{gray} \circle*{.5}}}	 \put(2,4){{\color{gray}\circle*{.5}}}	\put(3,4){{\color{gray}\circle*{.5}}} 	\put(4,4){{\color{gray}\circle*{.5}}}
\put(0,3){{\color{gray} \circle*{.5}}}	\put(1,3){{\color{gray} \circle*{.5}}}	 \put(2,3){{\color{gray}\circle*{.5}}}		
\put(0,2){{\color{gray} \circle*{.5}}}	\put(1,2){{\color{gray} \circle*{.5}}}	 \put(2,2){{\color{black}\circle*{.5}}}
\put(0,1){{\color{gray} \circle*{.5}}}	\put(1,1){{\color{gray} \circle*{.5}}}	 \put(2,1){{\color{gray}\circle*{.5}}}
\put(0,0){{\color{gray}\circle*{.5}}}
\put(-0.5,4.5){\line(1,-1){4}}
\multiput(-0.5,1.5)(3,0){2}{\line(0,1){3}}
\multiput(-0.5,1.5)(0,3){2}{\line(1,0){3}}

\put(7,4){{\color{gray} \circle*{.5}}}	\put(8,4){{\color{gray} \circle*{.5}}}	\put(9,4){{\color{gray}\circle*{.5}}}	\put(10,4){{\color{gray}\circle*{.5}}}
\put(7,3){{\color{gray} \circle*{.5}}}	\put(8,3){{\color{gray} \circle*{.5}}}	\put(9,3){{\color{gray}\circle*{.5}}}	\put(10,3){{\color{gray}\circle*{.5}}}	
\put(7,2){{\color{gray} \circle*{.5}}}	\put(8,2){{\color{gray} \circle*{.5}}}	\put(9,2){{\color{gray}\circle*{.5}}}	\put(10,2){{\color{gray}\circle*{.5}}}
\put(7,1){{\color{gray} \circle*{.5}}}	\put(8,1){{\color{gray}\circle*{.5}}}		 
\put(7,0){{\color{gray}\circle*{.5}}}	
\put(6.5,4.5){\line(1,-1){4}}
\multiput(6.5,1.5)(3,0){2}{\line(0,1){3}}
\multiput(6.5,1.5)(0,3){2}{\line(1,0){3}}
\end{picture}
\end{center}
\caption{The left-hand side shows the Ferrers diagram of $\alpha = (5,3,3,3,1)$ with its Durfee square; the black dot on the diagonal indicates the fixed point.  The right-hand side shows the analogous image for $\beta = (4,4,4,2,1)$ which does not have a fixed point.}  \label{fig4}
\end{figure}

We now confirm the conjecture of Blecher and Knopfmacher with an exact expression for the amount by which $g(n)$ exceeds $f(n)$ for $n > 2$.

\begin{corollary} \label{bkhyp}
For all $n\ge 2$,
\begin{align}
g(n) - f(n) & = M(0,n) \label{crank0} \\
& = p(n) + 2\sum_{j\ge1} (-1)^j \, p\!\left(n - \frac{j(j+1)}{2}\right). \label{crank0p}
\end{align}
Thus $g(n) > f(n)$ for $n > 2$.  
\end{corollary}

\begin{proof}
The crank equality \eqref{crank0} follows from Theorem \ref{ffmc} and expression \eqref{crank0p} in terms of $p(n)$ is discussed by Hopkins, Sellers, and Stanton \cite[page 5]{hss}.  
Since there is at least one crank 0 partition of $n$ for each $n \ge 3$, Conjecture \ref{bkconj} is confirmed.
\end{proof}

See Table \ref{tab1} below for data on the number of crank 0 partitions.

It is interesting that Blecher and Knopfmacher's conjecture is closely related to a result of Shen who showed
\[ \#\{ \lambda \vdash n \mid \mex(\lambda) \equiv 1 \bmod 2\} \ge \#\{ \lambda \vdash n \mid \mex(\lambda) \equiv 0 \bmod 2\} \]
for $n > 1$ \cite[Theorem 1]{s}.  With the connections Hopkins and Sellers established between the mex and the crank, they gave a much simpler proof that
\[\#\{ \lambda \vdash n \mid \mex(\lambda) \equiv 1 \bmod 2\} > \#\{ \lambda \vdash n \mid \mex(\lambda) \equiv 0 \bmod 2\}\]
for $n > 2$ \cite[Corollary 3]{hs}, a result equivalent to Corollary \ref{bkhyp}.


\section{Generalizing fixed points and the positive case}

Now we generalize Blecher and Knopfmacher's fixed point partition statistic.

\begin{definition} \label{kfp}
Given an integer $k$, a partition $\lambda$ has a $k$-fixed point if there is an index $i$ for which $\lambda_i = i+k$.
\end{definition}

For example, $\alpha=(5,3,3,3,1)$ has a 1-fixed point since $\alpha_2 = 2+1$ and $\beta = (4,4,4,2,1)$ has a $-2$-fixed point since $\beta_4 = 4-2$.  The fixed point statistic of Definition \ref{fp} is the $k=0$ case.  As before, for each $k$, a partition has at most one $k$-fixed point.

In this section, we consider $k$-fixed points for $k \ge 0$.  Let $f_k(n)$ be the number of partitions of $n$ with a $k$-fixed point, $g_k(n)$ the number without.

Our next result generalizes Theorem \ref{ffmc}.  Notice that {\it \ref{gcrank}}  involving the crank is the initial case of two families of related results, {\it \ref{gkcrank}} and {\it \ref{gkcrank2}}  below.

\begin{theorem}  \label{k+ffmc}
For $k \ge 0$ and $n \ge 2$, the following quantities are equal.
\begin{enumerate}
\renewcommand{\theenumi}{(\roman{enumi})}
\setcounter{enumi}{8}
\item $f_k(n)$, the number of partitions of $n$ with a $k$-fixed point, \label{fkfixed}
\item $\#\{\lambda \vdash n \mid \text{the top row of the Frobenius symbol includes k}\}$, \label{fkFrob} 
\item  $\#\{ \lambda \vdash n \mid \mex_k(\lambda) -k \equiv 0 \bmod 2\}$, \label{fkmex} 
\item  $\displaystyle \sum_{m \ge k+1} M(m,n)$. \label{fkcrank} 
\end{enumerate}
Under the same condition, the following quantities are also equal.
\begin{enumerate}
\renewcommand{\theenumi}{(\roman{enumi})}
\setcounter{enumi}{13}
\item $g_k(n)$, the number of partitions of $n$ without a $k$-fixed point, \label{gkfixed}
\item $\#\{\lambda \vdash n \mid \text{the top row of the Frobenius symbol excludes k}\}$, \label{gkFrob}
\item $\#\{ \lambda \vdash n \mid \mex_k(\lambda) -k \equiv 1 \bmod 2\}$, \label{gkmex}
\item $\displaystyle \sum_{m \ge -k} M(m,n)$, \label{gkcrank} 
\item $\displaystyle \sum_{m \ge k} M(m,n+k)$. \label{gkcrank2}
\end{enumerate}
\end{theorem}

\begin{proof}
To connect {\it \ref{fkfixed}} and {\it \ref{fkFrob}}, suppose $\lambda$ has $\lambda_i = i+k$.  Since $\lambda_j \ge i+k$ for any $j < i$ and $\lambda_j \le i+k$ for any $j > i$, this means that the $k$-Durfee rectangle has $d_k(\lambda) = i$ rows.  The column $i$ entry in the top row of the Frobenius symbol for $\lambda$ is $\lambda_i - i = k$.  (Note that the Durfee square parameter $d(\lambda)$ must satisfy $d(\lambda) \ge d_k(\lambda)$ so that the Frobenius symbol has at least $i$ columns.)

Conversely, suppose a partition $\lambda$ with $d_k(\lambda) = i$ has Frobenius symbol including $k$ in the top row.  By the maximality of the $k$-Durfee rectangle, $\lambda_i = i+k$, so $\lambda$ is among the partitions counted by $f_k(n)$.

Complementarity gives the equality of {\it \ref{gkfixed}} and {\it \ref{gkFrob}}.

Before proceeding to the other equalities, we establish generating functions for $f_k(n)$ and $g_k(n)$.  As mentioned above, a partition $\lambda$ with $\lambda_i = i+k$ has a $k$-Durfee rectangle of size $i \times (i+k)$.  Let $\beta'$ be the part of the Ferrers diagram of $\lambda$ to the right of the $k$-Durfee rectangle; see Figure \ref{proofpic}.  As a partition, $\beta'$ has at most $i - 1$ parts, otherwise $\lambda_i > i + k$.  That means $\beta$, the conjugate of the $\beta'$ shown, consists of parts strictly less than $i$.  The partition $\alpha$ beneath the $k$-Durfee rectangle has largest part at most $i+k$ since $\lambda_i = i+k$.  Therefore
\begin{equation}
\sum_{n \ge 0} f_k(n) q^n =
\sum_{i \ge 1} \frac{q^{i(i+k)}}{(q;q)_{i+k}(q;q)_{i-1}}. \label{fkgen}
\end{equation}
Now an expression for all partitions of $n$ in terms of their $k$-Durfee rectangles is 
\begin{equation}
\sum_{n \ge 0} p(n) q^n = \sum_{i\ge 0} \frac{q^{i(i+k)}}{(q;q)_{i+k}(q;q)_{i}}.  \label{pkgen}
\end{equation}
(Note that this summation starts from $k=0$ since, for example, the 2-Durfee rectangle of the partition $(1,1)$ is $0 \times 2$.)
Therefore $g_k(n) = p(n) - f_k(n)$ has generating function
\begin{equation}
\sum_{n \ge 0} g_k(n) q^n =
\sum_{i \ge 0} \frac{q^{i^2+ik+i}}{(q;q)_{i+k}(q;q)_{i}}. \label{gkgen}
 \end{equation}

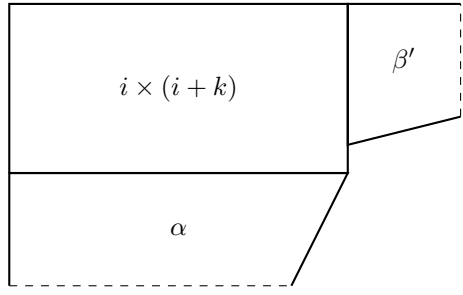
\begin{figure}
\begin{center}
\begin{tikzpicture}[scale=0.75]
\draw[thick] (0,0) -- (0,2) -- (6,2) -- (5,0);
\draw[thick] (0,2) -- (0,5) -- (6,5) -- (6,2) -- (0,2);
\draw[thick] (8,3) -- (6,2.5) -- (6,5) -- (8,5);
\draw[dashed] (8,3) -- (8,5);
\draw[dashed] (0,0) -- (5,0);
\node at (3,3.5) {$i \times (i+k)$};
\node at (3,1) {$\alpha$};
\node at (7,4) {$\beta'$};
\end{tikzpicture}
\end{center}
\caption{A partition with a $k$-fixed point decomposed by its $k$-Durfee rectangle.} \label{proofpic}
\end{figure}

To show that {\it \ref{gkfixed}} and {\it \ref{gkcrank}} are equal, note that by Theorem \ref{hsy8},
\begin{align*} 
\sum_{m \ge -k} \sum_{n \ge 0} M(m,n) q^n 
& = \sum_{i \ge k} \frac{q^{i^2-ik+i-k}}{(q;q)_i(q;q)_{i-k}}  \quad \textrm{(as we need $i-k\geq 0$)} \\ 
& = \sum_{h \ge 0} \frac{q^{h^2+hk+h}}{(q;q)_{h+k}(q;q)_h}  \quad \textrm{(letting $h = i-k$)}
\end{align*}
which is exactly \eqref{gkgen}.

Establishing the equality of {\it \ref{gkfixed}} and {\it \ref{gkcrank2}} is similar.  By Theorem \ref{hsy8}, 
\begin{align*} 
\sum_{m \ge k} \sum_{n \ge 0} M(m,n+k) q^{n+k} & = \sum_{i \ge 0} \frac{q^{i^2+ik+i+k}}{(q;q)_i(q;q)_{i+k}} \\ 
& = q^k \sum_{i \ge 0} \frac{q^{i^2+ik+i}}{(q;q)_i(q;q)_{i+k}} \\
&= \sum_{n \ge 0} g_k(n) q^{n+k}
\end{align*}
using \eqref{gkgen} again.

Finally, Theorem \ref{hss2} applied to partitions of $n+k$ shows that {\it \ref{gkcrank2}} and {\it \ref{gkmex}} are equal.  The equality of {\it \ref{fkfixed}}, {\it \ref{fkmex}}, and {\it \ref{fkcrank}} follows by complementarity and crank symmetry.
\end{proof}

See Figure \ref{fig5} for examples relating 1-fixed points to the 1-Durfee rectangle.

\begin{figure}
\begin{center}
\setlength{\unitlength}{.75cm}
\begin{picture}(11,5)
\thicklines
\put(0,4){{\color{gray} \circle*{.5}}}	\put(1,4){{\color{gray} \circle*{.5}}}	 \put(2,4){{\color{gray}\circle*{.5}}}	\put(3,4){{\color{gray}\circle*{.5}}} 	\put(4,4){{\color{gray}\circle*{.5}}}
\put(0,3){{\color{gray} \circle*{.5}}}	\put(1,3){{\color{gray} \circle*{.5}}}	 \put(2,3){{\color{black}\circle*{.5}}}		
\put(0,2){{\color{gray} \circle*{.5}}}	\put(1,2){{\color{gray} \circle*{.5}}}	 \put(2,2){{\color{gray}\circle*{.5}}}
\put(0,1){{\color{gray} \circle*{.5}}}	\put(1,1){{\color{gray} \circle*{.5}}}	 \put(2,1){{\color{gray}\circle*{.5}}}
\put(0,0){{\color{gray}\circle*{.5}}}
\put(0.5,4.5){\line(1,-1){3}}
\multiput(-0.5,2.5)(3,0){2}{\line(0,1){2}}
\multiput(-0.5,2.5)(0,2){2}{\line(1,0){3}}

\put(7,4){{\color{gray} \circle*{.5}}}	\put(8,4){{\color{gray} \circle*{.5}}}	\put(9,4){{\color{gray}\circle*{.5}}}	\put(10,4){{\color{gray}\circle*{.5}}}
\put(7,3){{\color{gray} \circle*{.5}}}	\put(8,3){{\color{gray} \circle*{.5}}}	\put(9,3){{\color{gray}\circle*{.5}}}	\put(10,3){{\color{gray}\circle*{.5}}}	
\put(7,2){{\color{gray} \circle*{.5}}}	\put(8,2){{\color{gray} \circle*{.5}}}	\put(9,2){{\color{gray}\circle*{.5}}}	\put(10,2){{\color{black}\circle*{.5}}}
\put(7,1){{\color{gray} \circle*{.5}}}	\put(8,1){{\color{gray}\circle*{.5}}}		 
\put(7,0){{\color{gray}\circle*{.5}}}	
\put(7.5,4.5){\line(1,-1){4}}
\multiput(6.5,1.5)(4,0){2}{\line(0,1){3}}
\multiput(6.5,1.5)(0,3){2}{\line(1,0){4}}
\end{picture}
\end{center}
\caption{The left-hand side shows $\alpha = (5,3,3,3,1)$ with its 1-Durfee square; the black dot on the diagonal indicates the fixed point.  The right-hand side shows the analogous image for $\beta = (4,4,4,2,1)$ which also has a 1-fixed point. Note the position of the diagonal lines relative to Figure \ref{fig4}.}  \label{fig5}
\end{figure}
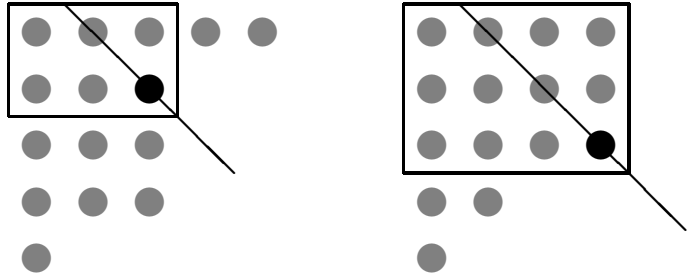

Note that one could also establish the equality of {\it \ref{gkcrank2}} and {\it \ref{gkFrob}} in Theorem \ref{k+ffmc} by another result of Hopkins, Sellers, and Stanton \cite[Theorem 7]{hss}, but we will need the generating function for $f_k(n)$ developed in this proof for the next section.

The analogue of Corollary \ref{bkhyp} shows that our resolution of Conjecture \ref{bkconj} is the first of an infinite family of enumeration results for the number of partitions of $n$ with crank in a range centered at 0.

\begin{corollary} \label{posrange}
For $k \ge 0$ and $n \ge 2$, 
\begin{align}
 g_k(n) - f_k(n) & = \sum_{m=-k}^k M(m,n) \label{crankk} \\
 & = p(n) + 2 \sum_{j \ge1} (-1)^j \, p\!\left( n - \frac{j(j+2k-1)}{2} \right). \label{crankkp} 
 \end{align}
\end{corollary}

\begin{proof}
The crank equality \eqref{crankk} follows from Theorem \ref{k+ffmc} {\it \ref{fkcrank}} and {\it \ref{gkcrank}}.  The expression \eqref{crankkp} in terms of $p(n)$ follows from Hopkins, Sellers, and Stanton \cite[Corollary 4]{hss}. 
\end{proof}

Table \ref{tab1} gives the number of partitions $\lambda$ of $n$ with $-k \le \crank{\lambda} \le k$ for small $n$ and $k$.  Note that for fixed $n$ and sufficiently large $k$, the range of crank values includes all partitions of $n$, thus each column eventually stabilizes at $p(n)$.  As of this writing, of these sequences, only the numbers of crank 0 partitions of $n$ are included in the OEIS \cite[A064410]{o}.

\begin{table}
\centering
\begin{tabular}{c|cccccccccccccc}
$k$\textbackslash$n$ & 2 & 3 & 4 & 5 & 6 & 7 & 8 & 9 & 10 & 11 & 12 & 13 & 14 & 15 \\ \hline
0 & 0 & 1 & 1 & 1 & 1 & 1 & 2 & 2 & 4 & 4 & 7 & 7 & 11 & 12 \\
1 & 0 & 1 & 1 & 3 & 3 & 5 & 6 & 8 & 10 & 14 & 17 & 23 & 29 & 38 \\
2 & 2 & 1 & 3 & 3 & 5 & 7 & 10 & 12 & 18 & 22 & 29 & 37 & 49 & 60 \\
3 & 2 & 3 & 3 & 5 & 7 & 9 & 12 & 18 & 22 & 30 & 39 & 51 & 65 & 84 \\
4 & 2 & 3 & 5 & 5 & 9 & 11 & 16 & 20 & 28 & 36 & 49 & 61 & 81 & 102 \\
5 & 2 & 3 & 5 & 7 & 9 & 13 & 18 & 24 & 32 & 42 & 55 & 73 & 93 & 120
\end{tabular}
\caption{Values of $\#\{ \lambda \vdash n \mid -k \le \crank(\lambda) \le k\}$ for $0 \le k \le 5$ and $2 \le n \le 15$.} \label{tab1}
\end{table}

\section{The negative case of generalized fixed points}

Definition \ref{kfp} allows for generalized fixed points with a negative parameter.  Set $k \ge 0$.  In this section, we focus on $-k$-fixed points.  There are important differences from fixed points with a positive parameter.

First, a partition $\lambda$ with a $-k$-fixed point must have at least $k+1$ parts since the condition $\lambda_i = i-k$ requires $i \ge k+1$ for the part $\lambda_i$ to be a positive integer.  As above, let $f_{-k}(n)$ be the number of partitions of $n$ with a $-k$-fixed point.

Second, there are subtypes within the complementary set.  Certainly, $\alpha = (5,3,3,3,1)$ does not have a $-2$-fixed point while $\beta = (4,4,4,2,1)$ does since $\beta_4 = 4-2$.  The partition $\gamma = (3)$ does not have a $-2$-fixed point, either, although one could argue that it does under the relaxation of various definitions: It is not unusual to say that a partition has addition parts 0, e.g., $\gamma_1 = 3$, $\gamma_2 = 0$, $\gamma_3 = 0$, etc.  Then $\gamma_2 = 2-2$ does seem to be a sort of $-2$-fixed point.  See Figure \ref{fig6}.  Rather than say $\gamma$ has a $-2$-fixed point, we separate partitions into three classes with regard to $-k$-fixed points.

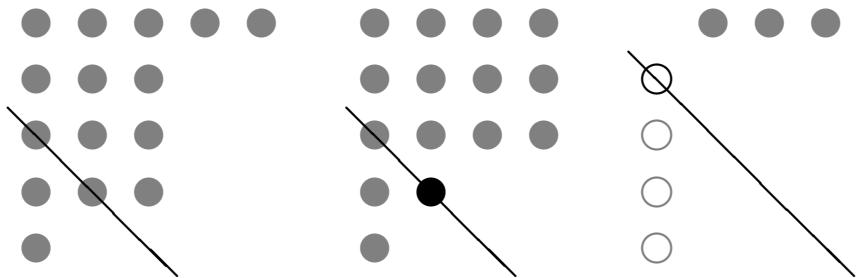
\begin{figure}[b!]
\begin{center}
\setlength{\unitlength}{.75cm}
\begin{picture}(14,5)
\thicklines
\put(0,4){{\color{gray} \circle*{.5}}}	\put(1,4){{\color{gray} \circle*{.5}}}	 \put(2,4){{\color{gray}\circle*{.5}}}	\put(3,4){{\color{gray}\circle*{.5}}} 	\put(4,4){{\color{gray}\circle*{.5}}}
\put(0,3){{\color{gray} \circle*{.5}}}	\put(1,3){{\color{gray} \circle*{.5}}}	 \put(2,3){{\color{gray}\circle*{.5}}}		
\put(0,2){{\color{gray} \circle*{.5}}}	\put(1,2){{\color{gray} \circle*{.5}}}	 \put(2,2){{\color{gray}\circle*{.5}}}
\put(0,1){{\color{gray} \circle*{.5}}}	\put(1,1){{\color{gray} \circle*{.5}}}	 \put(2,1){{\color{gray}\circle*{.5}}}
\put(0,0){{\color{gray}\circle*{.5}}}
\put(-0.5,2.5){\line(1,-1){3}}

\put(6,4){{\color{gray} \circle*{.5}}}	\put(7,4){{\color{gray} \circle*{.5}}}	\put(8,4){{\color{gray}\circle*{.5}}}	\put(9,4){{\color{gray}\circle*{.5}}}
\put(6,3){{\color{gray} \circle*{.5}}}	\put(7,3){{\color{gray} \circle*{.5}}}	\put(8,3){{\color{gray}\circle*{.5}}}	\put(9,3){{\color{gray}\circle*{.5}}}	
\put(6,2){{\color{gray} \circle*{.5}}}	\put(7,2){{\color{gray} \circle*{.5}}}	\put(8,2){{\color{gray}\circle*{.5}}}	\put(9,2){{\color{gray}\circle*{.5}}}
\put(6,1){{\color{gray} \circle*{.5}}}	\put(7,1){{\color{black}\circle*{.5}}}		 
\put(6,0){{\color{gray}\circle*{.5}}}	
\put(5.5,2.5){\line(1,-1){3}}

\put(12,4){{\color{gray} \circle*{.5}}}	\put(13,4){{\color{gray} \circle*{.5}}}	\put(14,4){{\color{gray}\circle*{.5}}}
\put(11,3){{\color{black} \circle{.5}}}
\put(11,2){{\color{gray} \circle{.5}}}
\put(11,1){{\color{gray} \circle{.5}}}
\put(11,0){{\color{gray} \circle{.5}}}
\put(10.5,3.5){\line(1,-1){4}}
\end{picture}
\end{center}
\caption{The left-hand side shows $\alpha = (5,3,3,3,1)$ which does not have a $-2$-fixed point.  The middle shows that $\beta = (4,4,4,2,1)$ does have a $-2$-fixed point indicated by the black dot. The right-hand side shows $\gamma=(3)$ with circles to the left of the standard Ferrers diagram to indicate implicit parts 0.  Note the position of the diagonal lines relative to Figure \ref{fig4} and Figure \ref{fig5}.}  \label{fig6}
\end{figure}

Let $g'_{-k}(n)$ count the partitions $\lambda \vdash n$ with at least $k$ parts for which $\lambda_i \ne i - k$ for all $i$.  Write $p(n,k-1)$ for the partitions of $n$ with at most $k-1$ parts.  We have
\[p(n) = f_{-k}(n) + p(n,k-1) + g'_{-k}(n).\]

It may be surprising that the $-k$ case is not more similar to the situation of Section 3.  Part of the issue is the asymmetry of the fixed point definition: It relates only to the subpartition to the right of the Durfee square, not the subpartition below it.  When the $k$-Durfee rectangle is taller than it is wide, equivalently when the diagonal identifying fixed points in the figures starts below the first row, the ambiguity about partitions with up to $k-1$ parts arises.

A further justification for this more subtle grouping is the following result.

\begin{theorem} \label{k-fc}
For $k \ge 0$ and $n \ge 2$,
\begin{enumerate}
\renewcommand{\theenumi}{(\roman{enumi})}
\setcounter{enumi}{17}
\item $\displaystyle f_{-k}(n) + p(n,k-1)  = \sum_{m \ge -k+1} M(m,n)$, \label{fk-crank}
\item $\displaystyle g'_{-k}(n) = \sum_{m\ge k} M(m,n)$. \label{gk-crank}
\end{enumerate}
\end{theorem}

\begin{proof}
The derivation of the $f_{k}(n)$ generating function in the proof of Theorem \ref{k+ffmc} still applies with the modification that $i-k$ needs to be at least 1, i.e.,
\[\sum_{n \ge 0} f_{-k}(n)q^n = \sum_{i-k\geq 1} \frac{q^{i(i-k)}}{(q;q)_{i-k}(q;q)_{i-1} }.\]
The generating function for $p(n,k-1)$ is $1/(q;q)_{k-1}$ so that the left-hand side of {\it \ref{fk-crank}} has generating function
\begin{align*}
\sum_{i-k\geq 1} \frac{q^{i(i-k)}}{(q;q)_{i-k}(q;q)_{i-1} } & + \frac{1}{(q;q)_{k-1}} \\
&= \sum_{h\geq 1} \frac{q^{h(h+k)}}{(q;q)_{h}(q;q)_{h+k-1} } + \frac{1}{(q;q)_{k-1}} \quad \textrm{(letting $h = i-k$)} \\
& = \sum_{h\geq 0} \frac{q^{h(h+k)}}{(q;q)_{h}(q;q)_{h+k-1} } 
\end{align*}
since the second term in the penultimate equality is just the $h=0$ term of the last summation. 

By Theorem \ref{hsy8}, the generating function for the right-hand side of {\it \ref{fk-crank}} is
\begin{align*} 
\sum_{m \ge -k+1} \sum_{n \ge 0} M(m,n) q^n 
& = \sum_{j \ge k-1} \frac{q^{(j+1)(j-k+1)}}{(q;q)_j(q;q)_{j-k+1}}  \quad \textrm{(as we need $j-k+1\geq 0$)} \\ 
& = \sum_{\ell \ge 0} \frac{q^{\ell(\ell+k)}}{(q;q)_{\ell}(q;q)_{\ell+k-1}}  \quad \textrm{(letting $\ell = j-k+1$)}
\end{align*}
and we conclude that the two sides of {\it \ref{fk-crank}} are equal.

The result {\it \ref{gk-crank}} follows by complementarity and crank symmetry.
\end{proof}

Note that there is no longer a direct interpretation in terms of the generalized mex or entries of the Frobenius symbol for $-k$-fixed points.

The analogue of Corollary \ref{posrange} shows that the crank ranges of Table \ref{tab1} describe differences for $-k$-fixed points as well.

\begin{corollary} \label{negrange}
For $k \ge 0$ and $n \ge 2$, 
\begin{align}
f_{-k}(n) + p(n,k-1) &- g'_{-k}(n)  = \sum_{m=-k+1}^{k-1} M(m,n) \label{crank-k} \\
 & = p(n) + 2 \sum_{j \ge1} (-1)^j \, p\!\left( n - \frac{j(j+2k-3)}{2} \right). \label{crank-kp} 
 \end{align}
\end{corollary}

\begin{proof}
The crank equality \eqref{crank-k} follows from Theorem \ref{k-fc}.  The expression \eqref{crank-kp} in terms of $p(n)$ follows from Hopkins, Sellers, and Stanton \cite[Corollary 4]{hss}. 
\end{proof}

Combining Corollary \ref{posrange} and Corollary \ref{negrange} gives the following curious identity connecting generalized fixed points with positive and negative parameters.
\begin{equation}
 g_{k-1}(n) - f_{k-1}(n) =  f_{-k}(n) + p(n,k-1) - g'_{-k}(n).
 \end{equation}

\section{Ideas for further investigations}

It would be interesting to have combinatorial proofs for the results proven here by analytic methods.  In particular, the equality of {\it \ref{gkcrank}} and {\it \ref{gkcrank2}} in Theorem \ref{k+ffmc}, namely that 
\begin{equation}
\# \{ \lambda \vdash n \mid \crank(\lambda) \ge -k\} = \# \{ \mu \vdash n+k \mid \crank(\mu) \ge k\},
 \end{equation}
 calls for a manipulation of the Ferrers graphs of partitions of $n$ with crank at least $-k$ to produce the Ferrers graphs of partitions of $n+k$ with crank at least $k$.

Also, are there partition statistics that can take the place of $\mex_j$ and the Frobenius symbol condition so that Theorem \ref{k-fc} about fixed points with a negative parameter includes as many different partition statistics as Theorem \ref{k+ffmc} about fixed points with a nonnegative parameter?

Finally, it has been fruitful to separate the odd mex partitions into those with mex congruent to 1 modulo 4 and those with mex congruent to 3 modulo 4 \cite{hss,hsy}.  Are there interesting ways to refine the partitions with (generalized) fixed points?

\end{document}